\documentclass{amsart}
\usepackage{amsthm}
\usepackage{amssymb,latexsym,amsmath,verbatim,enumerate,graphicx}
 \usepackage[colorlinks=true]{hyperref}
\hypersetup{urlcolor=blue, citecolor=red}
\usepackage{euscript,amsmath,amssymb,amsfonts,graphicx,bm} 
 \newtheorem{theorem}{Theorem}
\newtheorem{remark}{{\em Remark}}
\newtheorem{example}{{\em Example}}

\def\E{\mathbb{E}}
\def\P{\mathbb{P}}
\def\R{\mathbb{R}}
\def\<{\langle}
\def\>{\rangle}
\def\S{\mathcal{S}}

\def\w{\wedge}
\newcommand{\x}{{\bf x}}
\newcommand{\y}{{\bf y}}
\newcommand{\n}{{\bf n}}
\def\L{\mathcal{L}}
\def\M{\mathcal{M}}
\newcommand{\X}{{\bf X}}
\newcommand{\W}{{\bf W}}

\newcommand{\bbb}{{\bf b}}
\newcommand{\Markov}[2]{\underset{#1}{\overset{#2}{\rightleftharpoons}}}

\title[BVPS FOR STATISTICS OF DIFFUSION IN A RANDOM ENVIRONMENT]{Boundary value problems for statistics of diffusion in a randomly switching environment: PDE and SDE perspectives}

\author{Sean D. Lawley}
\thanks{Department of Mathematics, University of Utah, Salt Lake City, UT 84112 USA (lawley@math.utah.edu). The author was supported by NSF grant DMS-RTG 1148230.}

\begin{document}

\maketitle

\begin{abstract}
Driven by diverse applications, several recent models impose randomly switching boundary conditions on either a PDE or SDE. The purpose of this paper is to provide tools for calculating statistics of these models and to establish a connection between these two perspectives on diffusion in a random environment. Under general conditions, we prove that the moments of a solution to a randomly switching PDE satisfy a hierarchy of BVPs with lower order moments coupling to higher order moments at the boundaries. Further, we prove that joint exit statistics for a set of particles following a randomly switching SDE satisfy a corresponding hierarchy of BVPs. In particular, the $M$-th moment of a solution to a switching PDE corresponds to exit statistics for $M$ particles following a switching SDE. We note that though the particles are non-interacting, they are nonetheless correlated because they all follow the same switching SDE. Finally, we give several examples of how our theorems reveal the sometimes surprising dynamics of these systems.
\end{abstract}

\smallskip
\noindent \textbf{keywords.} random PDE, stochastic hybrid system, piecewise deterministic Markov process, randomly switching boundary, switched dynamical system.
\smallskip


\section{Introduction}\label{intro}

Several recent models impose randomly switching boundary conditions on either a partial differential equation (PDE) \cite{Bressloff15a, Lawley15neural1, Lawley15robin, Lawley15sima} or stochastic differential equation (SDE) \cite{Ammari11, Benichou11, Bressloff15b, Bressloff15c, Reingruber09, Reingruber10}. These models appear in a diverse set of fields, including neuroscience, insect physiology, medicine, biochemistry, intermittent search, and in the derivation of classical objects in dynamical systems. The PDE models arise from considering a \emph{density} of particles diffusing in a random environment, whereas the SDE models arise from considering only \emph{finitely many} particles diffusing in a random environment.

The purpose of this paper is to provide tools for calculating statistics of these models and to establish a connection between these two perspectives on diffusion in a random environment. We also give several examples to show how our tools elucidate the dynamics of these stochastic systems. In section~\ref{section partial} we consider evolution PDEs of the form
\begin{align*}
\partial_{t}u=\L u,
\end{align*}
and let both the boundary conditions and the differential operator $\L$ randomly switch according to a continuous-time Markov jump process. Under certain regularity assumptions on the resulting $L^{2}$-valued piecewise deterministic Markov process $\{u(\x,t)\}_{t\ge0}$, we prove that the moments of the process satisfy a hierarchy of boundary value problems (BVPs) with lower order moments coupling to higher order moments at the boundaries.

In section~\ref{section particle} we consider a set of particles diffusing in a bounded domain and allow both (a) the SDE governing their motion to randomly switch and (b) pieces of the boundary to switch between being either absorbing or reflecting. We note that though the particles are non-interacting, they are nonetheless correlated because they all follow the same switching SDE and boundary conditions. We prove that joint exit statistics of the particles (survival probabilities, mean first passage times, hitting probabilities) satisfy a hierarchy of BVPs that is very similar to the one for moments of solutions to switching PDEs. In particular, the $M$-th moment of a solution to a switching PDE corresponds to exit statistics for $M$ particles following a switching SDE.


We now comment on how this paper relates to recent work on similar systems. To our knowledge, \cite{Lawley15sima} and \cite{Bressloff15a} are the only other works that develop tools to analyze PDEs with randomly switching boundary conditions. The machinery of \cite{Lawley15sima} has the advantage that it does not require the switching to be Markovian, but our present results allow for much simpler calculation of statistics if the switching is Markovian. The BVPs of Theorem~\ref{general m} are derived in \cite{Bressloff15a} for a specific example by discretizing space and constructing the Chapman-Kolmogorov equation for the resulting finite-dimensional stochastic hybrid system. We generalize this result by using different techniques. Exit statistics for a single diffusing particle in a domain with switching boundary conditions are computed in \cite{Ammari11,  Bressloff15b, Bressloff15c, Reingruber09, Reingruber10}. Hitting probabilities for multiple particles are derived in \cite{Bressloff15a} for a specific example. To our knowledge, the present work is the first to give tools to compute joint exit statistics such as hitting probabilities, mean first passage times, and survival probabilities for multiple particles in general systems. Finally, the connections that we establish between switching PDEs and switching SDEs further develops the connections between classical potential theory and Brownian motion, first investigated over 60 years ago by Kakutani, Kac, and Doob \cite{doob54, kac51, kakutani44}.


In what follows, section~\ref{section partial} focuses on switching PDEs and section~\ref{section particle} considers switching SDEs. In addition to general theorems, both sections contain examples to (a) illustrate the biological applications that prompted this work and (b) show how the theorems reveal the dynamics of the stochastic systems (which are sometimes surprising). We conclude with a brief discussion and an appendix that collects some technical points.

\section{PDE perspective}\label{section partial}

\subsection{Hierarchy of moment equations}
In this section, we prove a theorem that gives the BVPs satisfied by the moments of the solution to a randomly switching PDE. Let $U\subset\R^{d}$ be an open set which will serve as the domain for our switching PDE (the regularity of $U$ will be handled by Assumptions \ref{continuous space}-\ref{neumann bound} below). For each $i$ in some finite set $I$, suppose we are given a differential operator of the form
\begin{align}\label{L}
\L_{i}u(\x) & = \sum_{j=1}^{d}(\bbb_{i})_{j}(\x)\partial_{x_{j}}u+\frac12\sum_{l,j=1}^{d}(\sigma_{i}\sigma_{i}^{T})_{l,j}(\x)\partial_{x_{l},x_{j}}u+s_{i}(\x)u,
\end{align}
where $\bbb_{i}:\bar{U}\mapsto\R^{d}$, $\sigma_{i}:\bar{U}\mapsto\R^{d\times d}$, $s_{i}:\bar{U}\mapsto\R$, and $x_{j}$ denotes the $j$-th component of $\x\in\R^{d}$. The three terms in (\ref{L}) respectively describe advection, diffusion, and any sinks or sources. In order to describe boundary conditions, suppose that for each $i\in I$, the pairwise disjoint sets $\Gamma_{i}^{\text{Dir}}$, $\Gamma_{i}^{\text{Neu}}$, and $\Gamma_{i}^{\text{Rob}}$ partition the boundary
\begin{align*}
\Gamma_{i}^{\text{Dir}}\cup\Gamma_{i}^{\text{Neu}}\cup\Gamma_{i}^{\text{Rob}}=\partial U,
\end{align*}
and we are given functions
\begin{align}
\begin{split}\label{BC functions}
g^{\text{Dir}}_{i}:\Gamma_{i}^{\text{Dir}}\to\R,
\qquad
g^{\text{Neu}}_{i}:\Gamma_{i}^{\text{Neu}}\to\R,\\
\text{and}\qquad
g^{\text{Rob}}_{i},h^{\text{Rob}}_{i}:\Gamma_{i}^{\text{Rob}}\to\R,\;f^{\text{Rob}}_{i}:\Gamma_{i}^{\text{Rob}}\to\partial U.
\end{split}
\end{align}
Dirichlet, Neumann, and Robin conditions will be imposed respectively on $\Gamma_{i}^{\text{Dir}}$, $\Gamma_{i}^{\text{Neu}}$, $\Gamma_{i}^{\text{Rob}}$ using the corresponding superscripted functions in (\ref{BC functions}). We note that $f_{i}^{\text{Rob}}$ is needed to describe non-local boundary conditions (see Example~\ref{thermostat}).

To describe the random switching, let $\{J(t)\}_{t\ge0}$ be a continuous time Markov jump process on $I$ with generator $Q$. The generator $Q$ is an $I\times I$ matrix with nonnegative off diagonal entries $q_{ij}\ge0$ giving the jump rate from state $i\in I$ to $j\in I$. The diagonal entries of $Q$ are chosen so that $Q$ has zero row sums and thus correspond to (minus) the total rate or leaving state $i\in I$.

Suppose that there exists a stochastic process $\{u(\x,t)\}_{t\ge0}$ adapted\footnote{Informally, a stochastic process $\{X(t)\}_{t\ge0}$ is adapted to a stochastic process $\{Y(t)\}_{t\ge0}$ if the value of $X(s)$ depends only on $\{Y(t)\}_{t\le s}$.} to $\{J(t)\}_{t\ge0}$ taking values in $L^{2}(U)$ satisfying the following properties.

\begin{enumerate} 
\item\label{continuous space}
For each $t>0$, we have that $u(\cdot,t)\in C^{2}(\bar{U})$ almost surely. That is, for each $t>0$ the spatial derivative $D^{\alpha}u(\x,t)$ extends continuously to $\bar{U}$ almost surely\footnote{If a property holds almost surely, then it holds except perhaps on an event of probability zero. Probability zero events do not affect statistics.} for each multi-index $\alpha$ satisfying $|\alpha|\le2$.
\item\label{continuous time}
For each $\x\in U$, we have that $u(\x,\cdot)\in C((0,\infty))$  almost surely.
\item\label{satisfies}
If $t>0$ and $\x\in U$, then we have that 
\begin{align*}
\partial_{t}u(\x,t)=\L_{J(t)}u(\x,t)\quad\text{almost surely.}
\end{align*}
\item\label{boundary assumption}
If $t>0$ and  $\partial_{\n}$ denotes the normal derivative, then almost surely we have
\begin{alignat*}{2}
u(\x,t)&=g^{\text{Dir}}_{J(t)}(\x),\quad &&\text{if }\x\in\Gamma_{J(t)}^{\text{Dir}},\\
\partial_{\n}u(\x,t)&=g^{\text{Neu}}_{J(t)}(\x),\quad &&\text{if }\x\in\Gamma_{J(t)}^{\text{Neu}},\\
u(f^{\text{Rob}}_{J(t)}(\x),t)+h^{\text{Rob}}_{J(t)}(\x)\partial_{\n}u(\x,t)&=g^{\text{Rob}}_{J(t)}(\x),\quad &&\text{if }\x\in\Gamma_{J(t)}^{\text{Rob}}.
\end{alignat*}
\item\label{C bound}
There exists a deterministic function $C:U\times(0,\infty)\to\R$ that is bounded on compact subsets such that if $\x\in U$, $t>0$, and $i\in I$, then almost surely
\begin{align*}
\sum_{j=1}^{d}\big(|(\bbb_{i})_{j}(\x)|+1\big)|\partial_{x_{j}}u(\x,t)|+\sum_{l,j=1}^{d}\big(|(\sigma_{i}\sigma^{T}_{i})_{l,j}(\x)|+1\big)|\partial_{x_{l},x_{j}}u(\x,t)|\\
+\big(|s_{i}(\x)|+1\big)|u(\x,t)|
\le C(\x,t).
\end{align*}
\item\label{neumann bound}
For each $t>0$ and $i\in I$, there exists a random variable $C_{2}(t)>0$ with finite expectation and a neighborhood $B$ of $\Gamma^{\text{Neu}}_{i}\cup\Gamma^{\text{Rob}}_{i}$ such that if $j\in\{1,\dots,d\}$ and $\x\in B$, then
\begin{align*}
|\partial_{x_{j}}u(\x,t)1_{J(t)=i}|\le C_{2}(t)
\quad\text{almost surely.}
\end{align*}
\end{enumerate}

Assumption~\ref{continuous space} allows us to define the pointwise process $\{u(\x,t)\}_{t\ge0}$ taking values in $\R$ for each $\x\in\bar{U}$, and Assumption~\ref{continuous time} asserts that this process is continuous in time for each $\x$ in the interior. Assumption~\ref{satisfies} asserts that the process does in fact satisfy a switching PDE and Assumption~\ref{boundary assumption} gives the switching boundary conditions. The bound in Assumption~\ref{C bound} allows us to exchange differentiation with expectation in the interior of the domain, and the bound in Assumption~\ref{neumann bound} allows us to exchange differentiation with expectation at the boundary. Dynamically, these assumptions ensure that the spatial variation in the random solution is bounded in the interior of the domain and has finite mean near the boundary.

These assumptions are satisfied if we choose a sufficiently regular domain and sufficiently regular differential operators and boundary conditions. For our motivating biological models (Examples~\ref{neural example 1}-\ref{insect example}), we consider the diffusion equation on an interval so that verifying Assumptions~\ref{continuous space}-\ref{C bound} follows from elementary properties of solutions to the diffusion equation such as smoothness and the maximum principle. For these examples, Assumption~\ref{neumann bound} can be verified by analyzing the spectral decompositions of the associated solution operators (see the Appendix).

For $M$ a positive integer, let $\x_{1},\dots,\x_{M}\in \bar{U}$, $i\in I$, and $t>0$. Define the function $v^{M}_{i}:\bar{U}^{M}\times[0,\infty)\to\R$ by
\begin{align}\label{v definition}
v^{M}_{i}(\x_{1},\dots,\x_{M},t) & := \E\Big[1_{J(t)=i}\prod_{m=1}^{M} u(\x_{m},t)\Big],
\end{align}
where $\E$ denotes pointwise expectation and $1_{A}$ denotes the indicator function on an event $A$. Define $v_{i}^{0}:=\P(J(t)=i)$.

Observe that summing the $v_{i}^{M}$'s over $i\in I$ gives
\begin{align*}
\sum_{i\in I}v_{i}^{M}(\x_{1},\dots\x_{M},t)=\E[u(\x_{1},t)\cdot\dots\cdot u(\x_{M},t)].
\end{align*}
Thus, if we let $\x_{1}=\dots=\x_{M}=\x$, then we obtain the $M$-th moment of $u(\x,t)$. The reason that we decompose the $M$-th moment into a sum of $v_{i}^{M}$'s is that the PDE for each $v_{i}^{M}$ involves $\L_{i}$ and the $v_{i}^{M}$ boundary conditions involve the $i$-th boundary conditions in (\ref{BC functions}). On the other hand, one cannot in general say the PDE and boundary conditions satisfied by the $M$-th moment of of $u(\x,t)$. This decomposition is key to determining statistics of $u(\x,t)$.

The following theorem gives the BVP satisfied by $\{v_{i}^{M}\}_{i\in I}$.

\begin{theorem}[Moments of randomly switching PDE]\label{general m}
If $t>0$ and $(\x_{1},\dots,\x_{M})\in U^{M}$, then
\begin{align}\label{430}
\begin{split}
\partial_{t}v_{i}^{M}
&=\Big(\sum_{m=1}^{M} \L^{m}_{i}\Big)v_{i}^{M}+\sum_{j\in I} q_{ji}v^{M}_j,
\end{split}
\end{align}
where $\L^{m}_{i}$ is the differential operator in (\ref{L}) acting on the $m$-th spatial variable $\x_{m}$, and $q_{ij}$ is the $(i,j)$-th entry of the generator $Q$ of $J(t)$.

If $t>0$, $\x_{m}\in\partial U$, $(\x_{1},\dots,\x_{m-1},\x_{m+1},\dots,\x_{M})\in U^{M-1}$, and $\partial_{\n_{m}}$ denotes the normal derivative with respect to the $m$-th spatial variable $\x_{m}$, then
\begin{alignat}{2}
 \frac{v_{i}^{M}(\x_{1},\dots,\x_{M},t)}{v_{i}^{M-1}(\x_{1},\dots,\x_{m-1},\x_{m+1},\dots,\x_{M},t)}&=g^{\text{Dir}}_{i}(\x_{m}),\quad &&\x_{m}\in\Gamma_{i}^{\text{Dir}},\label{431}\\
 \frac{\partial_{\n_{m}}v_{i}^{M}(\x_{1},\dots,\x_{M},t)}{v_{i}^{M-1}(\x_{1},\dots,\x_{m-1},\x_{m+1},\dots,\x_{M},t)}&=g^{\text{Neu}}_{i}(\x_{m}),\quad &&\x_{m}\in\Gamma_{i}^{\text{Neu}},\label{431n}
 \end{alignat}
 and
 \begin{align}
 \begin{split}
  \frac{v_{i}^{M}(\x_{1},\dots,\x_{m-1},f^{\text{Rob}}_{i}(\x_{m}),\x_{m+1},\dots,\x_{M},t)+h^{\text{Rob}}_{i}(\x_{m})\partial_{\n_{m}}v_{i}^{M}(\x_{1},\dots,\x_{M},t)}{v_{i}^{M-1}(\x_{1},\dots,\x_{m-1},\x_{m+1},\dots,\x_{M},t)}
 \\=g^{\text{Rob}}_{i}(\x_{m}),\quad \x_{m}\in\Gamma_{i}^{\text{Rob}}\label{431r}.
  \end{split}
\end{align}
\end{theorem}

\begin{remark}
{\rm In matrix notation, the PDE in (\ref{430}) is
\begin{align*}
\partial_{t}{\bf v}=(L+Q^{T}){\bf v},
\end{align*}
where ${\bf v}$ is the vector with $i$-th component $v_{i}^{M}$, $L$ is the diagonal matrix with $i$-th diagonal entry $\sum_{m=1}^{\M}\L_{i}^{m}$, and $Q^{T}$ is the transpose of the generator of $J(t)$.}
\end{remark}

\begin{remark}
{\rm If the randomly switching PDE only imposes Dirichlet conditions (that is, $\Gamma^{\text{Dir}}_{i}=\partial U$ for all $i\in I$), then Assumption~\ref{neumann bound} is superfluous.}
\end{remark}

\begin{proof}
We first prove (\ref{430}) for the case $M=1$. Fix $\x\in U$, $t>0$, and $i\in  I$. For each $h\in\mathbb{R}$, define the events
\begin{align*}
A^{t,h}_0&=\{\text{no jump times of }J\text{ in }[t,t+h]\}\\
A^{t,h}_1&=\{\text{one jump time of }J\text{ in }[t,t+h]\}\\
A^{t,h}_2&=\{\text{two or more jump times of }J\text{ in }[t,t+h]\},
\end{align*}
where $s$ is said to be a \emph{jump time} of $J$ is $\lim_{t\to s+}J(t)\ne\lim_{t\to s-}J(t)$. For ease of notation, for the remainder of the proof we use $\{A\}$ to denote the indicator function for an event $A$. Suppressing the $\x$ dependence, we have that
\begin{align}
\partial_{t} v^{1}_i(t)
& = \lim_{h\to0}\frac1h \E\big[\big(u(t+h)\{J(t+h)=i\}-u(t)\{J(t)=i\}\big)\{A^{t,h}_0\}\big]\nonumber\\
&\quad + \lim_{h\to0}\frac1h \E\big[\big(u(t+h)\{J(t+h)=i\}-u(t)\{J(t)=i\}\big)\{A^{t,h}_1\}\big]\nonumber\\
&=:\mathcal{T}_{0}+\mathcal{T}_{1},\label{2 terms}
\end{align}
by the bound in Assumption~\ref{C bound} and the fact that $\P(A^{t,h}_{2})=o(h)$. The fact that $\P(A^{t,h}_{2})=o(h)$ is fairly standard \cite{norris98} and follows from the fact that the time between jumps of a continuous-time Markov chain is exponentially distributed. We work on the two terms in (\ref{2 terms}) separately.

First, note that
\begin{align*}
\mathcal{T}_{0}
& = \lim_{h\to0} \E\big[\frac1h\big(u(t+h)-u(t)\big)\{A^{t,h}_0\}\{J(t)=i\}\big].
\end{align*}
Now, by Assumption~\ref{satisfies}, we have the following almost sure equality
\begin{align*}
\lim_{h\to0}\frac1h\big(u(t+h)-u(t)\big)\{A^{t,h}_0\}\{J(t)=i\}=\partial_{t}u(t)\{J(t)=i\}.
\end{align*}
Further, by the mean value theorem there exists a random $\xi(h)$ so that
\begin{align*}
\big|\frac1h\big(u(t+h)-u(t)\big)\{A^{t,h}_0\}\{J(t)=i\}\big|&=
|\partial_{t}u(\xi(h))\{A^{t,h}_0\}\{J(t)=i\}|\\
&=|\L_{i}u(\xi(h))\{A^{t,h}_0\}\{J(t)=i\}|\\
&\le \sup_{s\in(t-h,t+h)}C(\x,s)\quad\text{almost surely},
\end{align*}
by Assumption~\ref{satisfies} and the bound in Assumption~\ref{C bound}. Thus, by the bounded convergence theorem, we have that
\begin{align*}
\mathcal{T}_{0}=\E[\partial_{t}u(t)\{J(t)=i\}]=\E[\L_i u(t)\{J(t)=i\}],
\end{align*}
where the second equality follows by virtue of Assumption~\ref{satisfies}. It is then straightforward to use Assumption~\ref{C bound} and the bounded convergence theorem to exchange the differential operator $\L_{i}$ with the expectation to conclude
\begin{align}\label{Psi0}
\mathcal{T}_{0}=\L_{i}v_{i}^{1}(t).
\end{align}

Moving to the second term in (\ref{2 terms}), observe that
\begin{align}
&\E\big[u(t+h)\{J(t+h)=i\}\{A^{t,h}_1\}\big] - \E\big[u(t)\{J(t)=i\}\{A^{t,h}_1\}\big]\nonumber\\
\begin{split}\label{2 sums}
& = \sum_{j\ne i}\E\big[u(t+h)\{J(t+h)=i\}\{J(t)=j\}\{A^{t,h}_{1}\}\big]\\
&\quad- \sum_{j\ne i}\E\big[u(t)\{J(t+h)=j\}\{J(t)=i\}\{A^{t,h}_1\}\big].
\end{split}
\end{align}
Suppose $h>0$. Focusing on the second sum in (\ref{2 sums}), let $j\ne i$ and observe that
\begin{align}
\begin{split}\label{easy term}
&\E[u(t)\{J(t+h)=j\}\{J(t)=i\}\{A^{t,h}_{1}\}]\\
&\quad=\E[u(t)\{J(t+h)=j\}\{J(t)=i\}]+o(h)\\
&\quad = \P(J(t+h)=j|J(t)=i)\E[u(t)\{J(t)=i\}]+o(h)
 =  hq_{ij} v^{1}_i(t)+o(h),
\end{split}
\end{align}
by $\P(A^{t,h}_{2})=o(h)$, Assumption~\ref{C bound}, and $\P(J(t+h)=j|J(t)=i)= hq_{ij}+o(h)$ \cite{norris98}.

We would like to apply the same argument to the terms in the first sum in (\ref{2 sums}), but those terms contain a $u(t+h)$ instead of a $u(t)$ in the expectation. Fortunately, we can show that these are close to each other. If $j\ne i$ and the random $\sigma\in(0,h)$ is such that $t+\sigma$ is the jump time between $t$ and $t+h$, then
\begin{align}\label{sigma close}
\begin{split}
&\E[(u(t+h)-u(t))\{J(t+h)=i\}\{J(t)=j\}\{A^{t,h}_1\}]\\
& = \E[(u(t+h)-u(t+\sigma))\{J(t+h)=i\}\{J(t)=j\}\{A^{t,h}_1\}]\\
&\qquad + \E[(u(t+\sigma)-u(t))\{J(t+h)=i\}\{J(t)=j\}\{A^{t,h}_1\}]\\
& \le h\sup_{s\in(t-h,t+h)}C(\x,s)\P(J(t+h)=i\}\{J(t)=j))+o(h),
\end{split}
\end{align}
by Assumption~\ref{satisfies}, the mean value theorem, the bound in Assumption~\ref{C bound}, and $\P(A^{t,h}_{2})=o(h)$. Thus, combining~(\ref{easy term}) and (\ref{sigma close}), we have that if $j\ne i$, then
\begin{align}\label{harder term}
\E\big[u(t+h)\{J(t+h)=i\}\{J(t)=j\}\{A^{t,h}_1\}\big]
& =  hq_{ji} v^{1}_j(t)+o(h).
\end{align}

An analogous argument shows that (\ref{easy term}) and (\ref{harder term}) also hold for $h<0$. Putting this together, we have that
\begin{align}\label{Psi1}
\mathcal{T}_{1}
& = \sum_{j\ne i} q_{ji}v^{1}_j(t) - \sum_{j\ne i} q_{ij}v^{1}_i(t)=\sum_{j\in I}q_{ji}v_{j}^{1}(t).
\end{align}
Finally, combining (\ref{2 terms}), (\ref{Psi0}), and (\ref{Psi1}) verifies (\ref{430}) for $M=1$.

The proof of (\ref{430}) for $M>1$ is similar, and so we only sketch it. As before, we have that
\begin{align*}
\partial_{t} v^{M}_i(t)
 = \lim_{h\to0}\frac1h \E\Big[\Big(\{J(t+h)=i\}\prod_{m=1}^{M}u(\x_{m},t+h)-\{J(t)=i\}\prod_{m=1}^{M}u(\x_{m},t)\Big)\{A_{0}^{t,h}\}\Big]\\
\quad + \lim_{h\to0}\frac1h \E\Big[\Big(\{J(t+h)=i\}\prod_{m=1}^{M}u(\x_{m},t+h)-\{J(t)=i\}\prod_{m=1}^{M}u(\x_{m},t)\Big)\{A_{1}^{t,h}\}\Big]\\
=:\mathcal{T}^{M}_{0}+\mathcal{T}^{M}_{1}.
\end{align*}
The proof that
\begin{align*}
\mathcal{T}^{M}_{0}=\sum_{m=1}^{M} \L^{m}_{i}v_{i}^{M}(\x_{1},\dots,\x_{M},t)
\end{align*}
proceeds as the proof of (\ref{Psi0}) with the added complication that we must use the product rule for differentiation. The proof that 
\begin{align*}
\mathcal{T}^{M}_{1}=\sum_{j\in I} q_{ji}v^{M}_j(\x_{1},\dots,\x_{M},t)
\end{align*}
is the same as the proof of (\ref{Psi1}).

We now verify the boundary conditions. The Dirichlet condition in (\ref{431}) is immediate. To verify the Neumann and Robin conditions in (\ref{431n}) and (\ref{431r}), it is enough to show that if $\x_{k}\in\Gamma^{\text{Neu}}_{i}\cup\Gamma^{\text{Rob}}_{i}$, then
\begin{align*}
\partial_{\n_{k}}\E\Big[\{J(t)=i\}\prod_{m=1}^{M}u(\x_{m},t)\Big] = \E\Big[\{J(t)=i\}\partial_{\n_{k}} u(\x_{k},t)\prod_{m=1,m\ne k}^{M}u(\x_{m},t)\Big].
\end{align*}
This follows immediately from Assumption~\ref{neumann bound} and the dominated convergence theorem.
\qquad\end{proof}


\subsection{PDE examples}\label{PDE examples}
In this section, we apply Theorem~\ref{general m} to a series of examples. The purpose of this section is to give some of the biological applications that prompted this paper and to illustrate how Theorem~\ref{general m} can elucidate the dynamics of these stochastic PDEs. Checking that these examples satisfy the appropriate hypotheses is discussed in the Appendix.

\begin{example}[Neurotransmitter concentration]\label{neural example 1}

\begin{figure}[t!]
\centering
\includegraphics[width=1\linewidth]{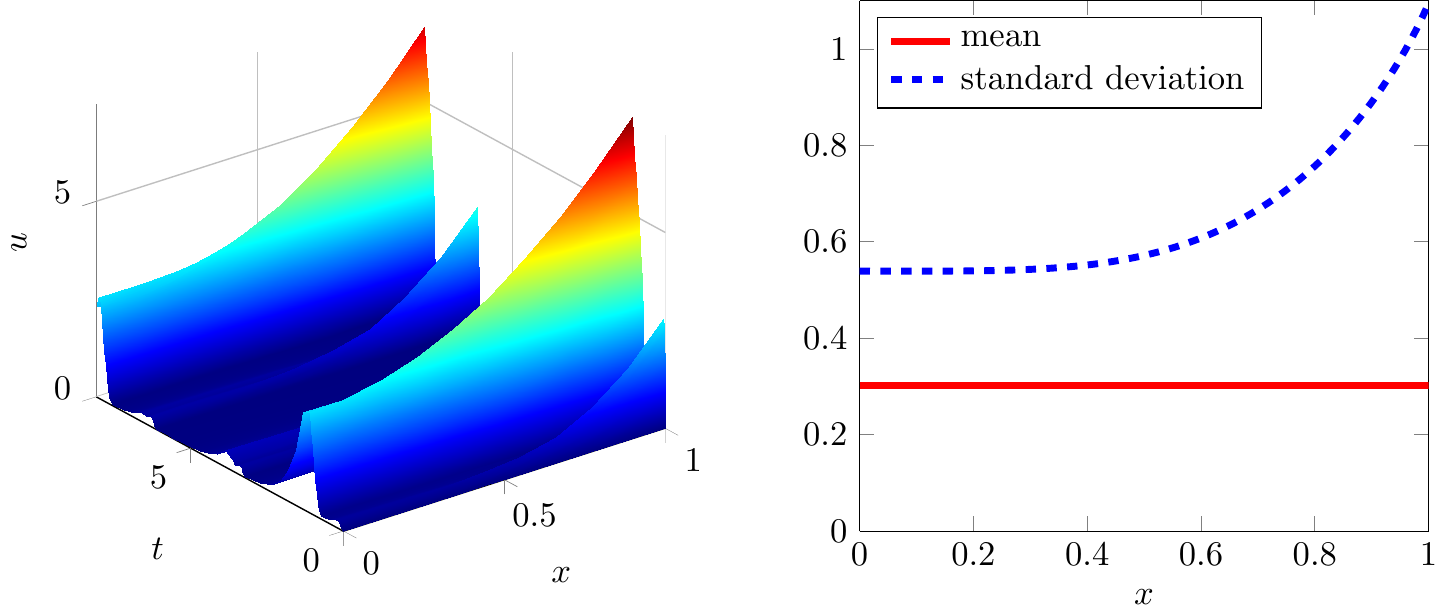}
\caption{Dynamics of Example~\ref{neural example 1}. On the left, we plot a single stochastic realization of Example~\ref{neural example 1}. From this realization, one can see that the stochastic solution fluctuates wildly in the region near the switching boundary at $x=1$ compared to the static boundary at $x=0$. This discrepancy is captured by the plot on the right, which shows that the large time mean solution is constant in space, but the large time standard deviation spikes at the switching boundary. The mean solution was found by explicitly solving (\ref{neural BVP}), and the variance was found by numerically solving (\ref{neural BVP2}). As a check, we also performed Monte Carlo simulations of the stochastic system and computed the empirical mean and variance. The empirical Monte Carlo curves were indistinguishable from the curves obtained from solving (\ref{neural BVP}) and (\ref{neural BVP2}). In both plots above, we take $L=D=\beta=1$ and $c=\alpha=10$.}
\label{figure neural}
\end{figure}

{\rm Sets of neurons can project to a distant region of the brain and trigger the release of neurotransmitter in that region by firing action potentials. This mechanism, known as volume transmission, enables groups of neurons to affect distant regions of the brain and is an important factor in motor control, Parkinson's disease, and the sleep/wake cycle \cite{fuxe, reed}.

As a prototype model for volume transmission, consider neurotransmitter diffusing in the interval $[0,L]$ with a single nerve terminal at $x=L$ that switches between a quiescent state and a firing state. Let $u(x,t)$ be the concentration of neurotransmitter in the interval $[0, L]$ and suppose $u$ satisfies the diffusion equation
\begin{align*}
\partial_{t}u & = D\Delta u \quad x\in(0,L),\;t>0,
\end{align*}
with $\partial_{x}u(0,t)=0$ and a condition at $x=L$ that randomly switches between
\begin{align*}
u(L,t)  = 0\text{ (quiescent neuron)}
\quad\Markov{\alpha}{\beta}\quad
\partial_{x}u(L,t)  = c>0\text{ (firing neuron)},
\end{align*}
with switching rates $\alpha$ and $\beta$. At $x=L$, the absorbing Dirichlet condition corresponds to the quiescent state (absorbing neurotransmitter), and the inhomogeneous Neumann condition corresponds to the firing state of the neuron (releasing neurotransmitter). Reference \cite{Lawley15neural1} appeals to Theorem~\ref{general m} analyze this model and more complicated models (with multiple neurons that fire independently) in order to understand certain aspects of volume transmission.

For this model, applying Theorem~\ref{general m} with $M=1$ shows that the steady state expected neurotransmitter, $\lim_{t\to\infty}\E[u(x,t)]$, is the sum $v_{0}(x)+v_{1}(x)$, where
\begin{align}\label{neural BVP}
\begin{split}
&\begin{pmatrix}0\\0\end{pmatrix}=D\Delta\begin{pmatrix}v_0\\v_1\end{pmatrix}+
\begin{pmatrix}
-\beta & \alpha\\
\beta & -\alpha
\end{pmatrix}
\begin{pmatrix}v_0\\v_1\end{pmatrix},\quad x\in(0,L),\\
&\partial_{x}v_{0}(0)=\partial_{x}v_{1}(0)=v_{0}(L)=0,
\quad
\partial_{x}v_{1}(L)=c\beta/(\alpha+\beta).
\end{split}
\end{align}
Solving this BVP explicitly, we find that the mean neurotransmitter concentration is constant in space. Furthermore, while the mean is constant, applying Theorem~\ref{general m} with $M=2$ reveals that the steady state 2-point correlation, $\lim_{t\to\infty}\E[u(x,t)u(y,t)]$, is the sum $v_{0}^{(2)}(x,y)+v_{1}^{(2)}(x,y)$, where
\begin{align}\label{neural BVP2}
\begin{split}
&\begin{pmatrix}0\\0\end{pmatrix}=D\Delta\begin{pmatrix}v_0^{(2)}\\v_1^{(2)}\end{pmatrix}+
\begin{pmatrix}
-\beta & \alpha\\
\beta & -\alpha
\end{pmatrix}
\begin{pmatrix}v_0^{(2)}\\v_1^{(2)}\end{pmatrix},\quad (x,y)\in(0,L)\times(0,L),\\
&\partial_{x}v^{(2)}_{0}(0,y)=\partial_{y}v^{(2)}_{0}(x,0)=\partial_{x}v^{(2)}_{1}(0,y)=\partial_{y}v^{(2)}_{1}(x,0)=0,\\
&v^{(2)}_{0}(L,y)=v^{(2)}_{0}(y,L)=0,
\quad
\partial_{x}v^{(2)}_{1}(L,y)=cv_{1}(y),\quad\partial_{y}v^{(2)}_{1}(x,L)=cv_{1}(x).
\end{split}
\end{align}
It is straightforward to solve this BVP numerically and obtain that the standard deviation of neurotransmitter spikes at $x=L$, despite the fact that the mean is constant in space (see Figure~\ref{figure neural}). Thus, the actual stochastic dynamics depend heavily on space, even though the mean dynamics do not. This discrepancy highlights the utility of using Theorem~\ref{general m} to calculate higher order moments. Using Theorem~\ref{general m} to analyze two and three-dimensional models will be the subject of future work.
}

\end{example}


\begin{figure}[t!]
\centering
\includegraphics[width=1\linewidth]{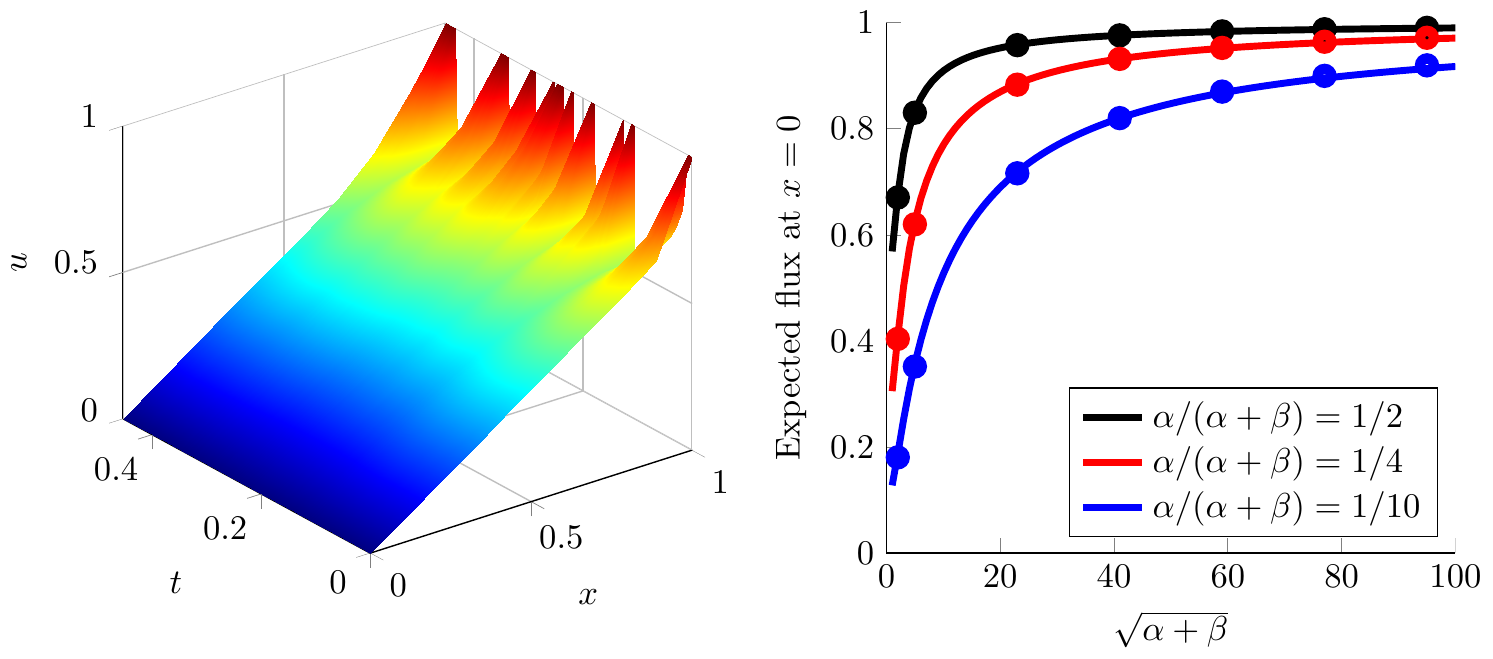}
\caption{Dynamics of Example~\ref{insect example}. On the left, we plot a single stochastic realization of Example~\ref{insect example}, and from this realization one can see that the slope of the solution at $x=0$ (and hence the flux) stays relatively close to 1, despite the fact that the boundary condition at $x=1$ is often no flux Neumann. The plot on the right captures this phenomena by plotting the large time expected flux at $x=0$ as a function of the switching rate for 3 different values of the proportion of time in the inhomogeneous Dirichlet state. We take $L=D=c=1$ so that if the boundary condition at $x=1$ was always inhomogeneous Dirichlet, then the flux at $x=0$ would be 1. Thus, these curves show that the flux at $x=0$ can remain high even when the proportion of time in the inhomogeneous Dirichlet state is small, provided the switching is fast. Biologically, this means that an insect can have its spiracles open a small proportion of time and yet receive essentially just as much oxygen as if its spiracles were always open, provided they open and close with high frequency. The curves in the right plot were found by explicitly solving the PDE in (\ref{neural BVP}) with boundary conditions in (\ref{insect BVP}). The dots in the right plot are empirical fluxes calculated from Monte Carlo simulations of the stochastic system. In both plots, we take $L=D=c=1$. In the left plot, we take $\alpha=25$ and $\beta=75$.}
\label{figure insect}
\end{figure}

\begin{example}[Insect respiration]\label{insect example}
{\rm Essentially all insects breathe through a network of tubes that allows oxygen and carbon dioxide to diffuse to and from their cells \cite{Wigglesworth31}. Air enters and exits this network through valves (called spiracles) in the exoskeleton, which sometimes regulate air flow by rapidly opening and closing. The purpose of this rapid opening and closing has perplexed physiologists for decades \cite{Chown06, Lighton96}. 

In order to explain this behavior, the following simple model was first proposed in \cite{Lawley15sima}. Let $u(x,t)$ be the oxygen concentration in a respiratory tube represented by the interval $[0, L]$, and so $u$ satisfies the diffusion equation
\begin{align*}
\partial_{t}u & = D\Delta u \quad x\in(0,L),\;t>0.
\end{align*}
We let $x=0$ represent where the tube meets the insect tissue, and so we impose an absorbing Dirichlet condition there, $u(0,t)=0$. The $x=L$ end represents the spiracle, and we suppose that the boundary condition there switches between
\begin{align*}
u(L,t)  = c>0\text{ (open spiracle)}
\quad\Markov{\alpha}{\beta}\quad
\partial_{x}u(L,t)  = 0\text{ (closed spiracle)},
\end{align*}
with switching rates $\alpha$ and $\beta$. When the spiracle is open we set $u(L,t)$ equal to the ambient oxygen concentration, and we impose a no flux condition when the spiracle is closed.

Applying Theorem~\ref{general m} to this model reveals that the steady state expected oxygen concentration, $\lim_{t\to\infty}\E[u(x,t)]$, is the sum $v_{0}(x)+v_{1}(x)$, where $v_{0}$ and $v_{1}$ satisfy the PDE in (\ref{neural BVP}) with boundary conditions
\begin{align}\label{insect BVP}
&v_{0}(0)=v_{1}(0)=\partial_{x}v_{0}(L)=0,
\quad
v_{0}(L)=c\alpha/(\alpha+\beta).
\end{align}
Solving this BVP explicitly yields the surprising result that an insect can maintain a large oxygen flux to its tissue while keeping its spiracles closed the vast majority of the time (see Figure~\ref{figure insect}). A forthcoming physiology paper will employ Theorem~\ref{general m} to further analyze this and more detailed models involving branching respiratory tubes \cite{lawley16pb5}.}

\end{example}


\begin{figure}[t!]
\centering
\includegraphics[width=1\linewidth]{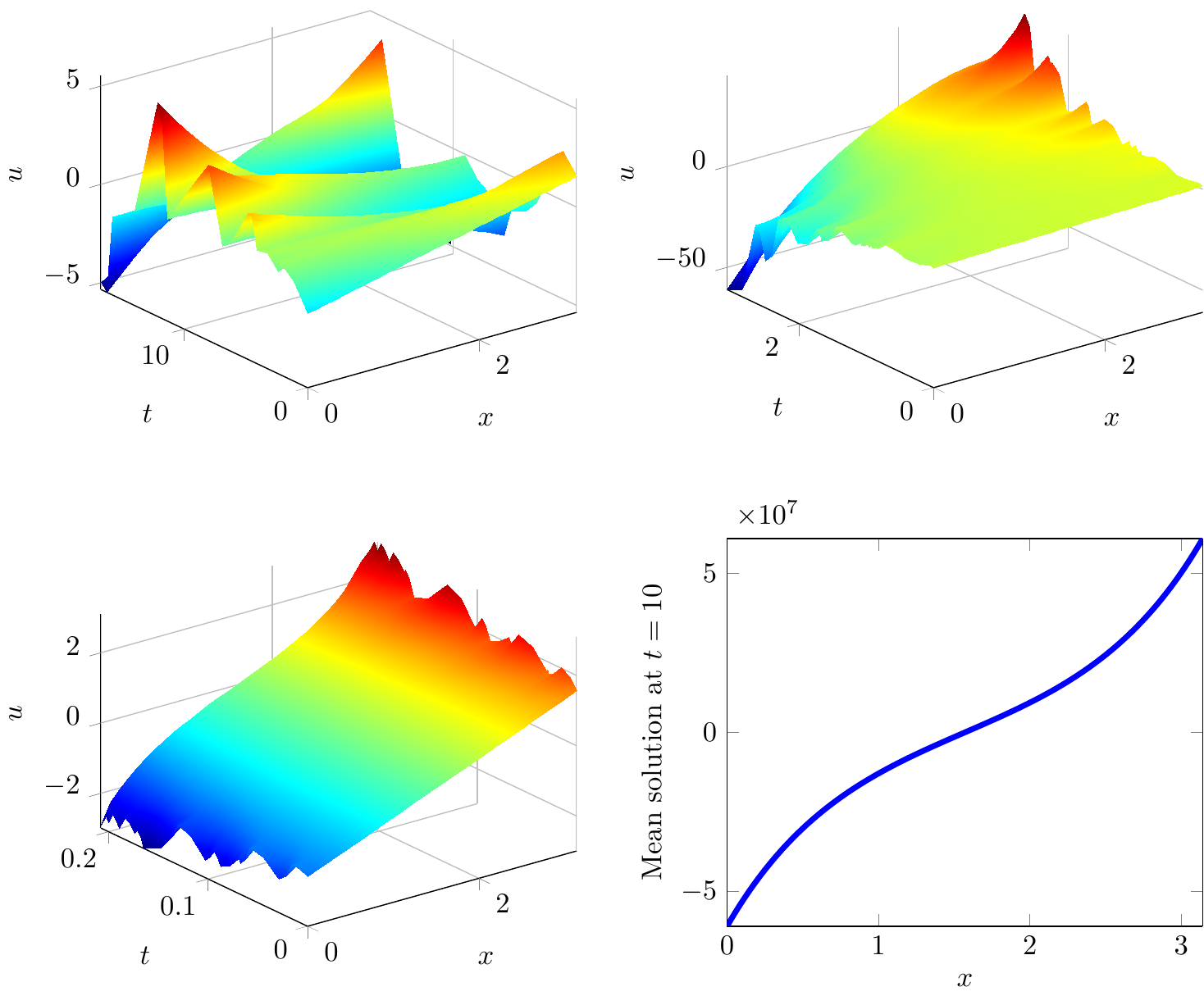}
\caption{Dynamics of Example~\ref{thermostat}. We plot single stochastic realizations of Example~\ref{thermostat} for increasing values of the switching rate $\alpha$ in the top left ($\alpha=1$), top right ($\alpha=10$), and bottom left ($\alpha=100$). In the bottom right, we plot the mean of the solution at time $t=10$ with $\alpha=100$. The initial condition is $u(x,0)=x-\pi/2$, and thus the mean is rapidly diverging. This last plot was obtained by numerically solving the BVP in Theorem~\ref{general m}. As a check, we also performed Monte Carlo simulations of the stochastic system and computed the empirical mean. This empirical mean Monte Carlo curve was indistinguishable from the curve obtained from Theorem~\ref{general m} and thus is not plotted. In all plots, we take $k=3$ which is below the critical threshold $k_{c}\approx5.6$. That is, for our choice of $k=3<k_{c}$, both individual systems vanish but the switched system blows up.}
\label{figure thermo}
\end{figure}

\begin{example}[Switching thermostat]\label{thermostat}
{\rm Having given some of the biological models that prompted this paper, we now give an example to show that the dynamics of a switching PDE can differ tremendously from the dynamics of the corresponding non-switching PDEs. Specifically, we give two sets of boundary conditions for the heat equation such that the solution converges to zero for each set of boundary conditions, but if the boundary conditions randomly switch, then Theorem~\ref{general m} reveals the surprising result that the solution goes to infinity.

Consider the following rudimentary model of a thermostat \cite{guidotti97, kalna04}. Suppose the temperature, $u(x,t)$, in the interval $[0,\pi]$ satisfies
\begin{align}\label{temp}
\partial_{t}u & = \Delta u \quad x\in(0,\pi),\;t>0.
\end{align}
To model an air conditioner located at $x=0$ and a thermostat located at $x=\pi$, we impose the boundary conditions
\begin{align}\label{temp0}
\partial_{x}u(0,t)  = k u(\pi,t)
\quad\text{and}\quad
\partial_{x}u(\pi,t)  = 0,
\end{align}
for some $k>0$. Notice that the flux at the air conditioner depends on the temperature at the thermostat. If the locations of the air conditioner and the thermostat were flipped, then we would impose
\begin{align}\label{temp1}
\partial_{x}u(0,t)  = 0
\quad\text{and}\quad
\partial_{x}u(\pi,t)  = -k u(0,t).
\end{align}
It can be shown that there exists a critical $k_{c}>0$ such that if $k\in(0,k_{c})$, then the solution to (\ref{temp}) with either boundary conditions (\ref{temp0}) or (\ref{temp1}) will vanish at large time for any initial condition \cite{guidotti97, kalna04}.

Now, suppose the boundary conditions randomly switch between (\ref{temp0}) and (\ref{temp1}) according to a continuous-time Markov jump process with jump rate $\alpha>0$. Then, if we suppose Assumptions~\ref{C bound}-\ref{neumann bound} are satisfied (see the Appendix for a discussion), then Theorem~\ref{general m} gives that the expected temperature, $\E[u(x,t)]$, is the sum $v_{0}(x,t)+v_{1}(x,t)$, where
\begin{align}
\begin{split}\label{thermostat BVP}
&\partial_t\begin{pmatrix}v_0\\v_1\end{pmatrix}=\Delta\begin{pmatrix}v_0\\v_1\end{pmatrix}+
\alpha\begin{pmatrix}
-1 & 1\\
1 & -1
\end{pmatrix}
\begin{pmatrix}v_0\\v_1\end{pmatrix},\quad x\in(0,\pi)\text{, }t>0,\\
&\partial_{x}v_{0}(0,t)=kv_{0}(\pi,t),
\quad
\partial_{x}v_{0}(\pi,t)=0,\\
&\partial_{x}v_{1}(0,t)=0,
\quad
\partial_{x}v_{1}(\pi,t)=-kv_{1}(0,t).
\end{split}
\end{align}
By analyzing this BVP, one can show that there exists $\alpha>0$ and $k\in(0,k_{c})$ such that $\E[u(x,t)]\to\infty$ in the $L^{1}[0,\pi]$ norm. Thus, stochastically switching between two stable PDE BVPs induces a blowup. Figure~\ref{figure thermo} plots $\E[u(x,t)]$ in this parameter regime where both individual systems vanish but the switched system blows up. A detailed bifurcation analysis of (\ref{thermostat BVP}) and the underlying stochastic system will be included in a forthcoming publication. Such a blowup is reminiscent of stochastically switched linear ODEs that blowup despite switching between only stable systems \cite{benaim14, lawley14}. 

}
\end{example}

\begin{example}[Deriving Robin boundary and interface jump conditions]\label{robin example}{\rm It was recently shown that the classical Robin boundary condition and interface jump condition can be derived as averages of certain switching conditions \cite{Lawley15robin}. By analyzing the BVP given by Theorem~\ref{general m} for the mean of the switching PDE in Example~\ref{insect example}, it was shown that the mean of the solution with the switching condition converges to a solution with a Robin condition in a certain fast switching limit. That is, switching between Dirichlet and Neumann conditions produces a Robin condition if the switching rate goes to infinity \emph{and} the proportion of time in the Dirichlet state goes to 0 at a corresponding rate. If, however, the proportion of time in the Dirichlet state is bounded away from zero, then one obtains pure Dirichlet as the switching rate goes to infinity.

To derive an interface jump condition, suppose that $u(x,t)$ satisfies the diffusion equation on $[0,L]$ with deterministic boundary conditions (say, $u(0,t)=0$ and $u(L,t)=c>0$), but with a randomly imposed no flux condition at $x=L/2$:
\begin{align*}
J(t)\partial_{x}u(L/2,t) = 0,
\end{align*}
where $J(t)\in\{0,1\}$ is a continuous-time Markov jump process
\begin{align*}
0\Markov{\alpha}{\beta}1,
\end{align*}
with switching rates $\alpha$ and $\beta$. If $J(t)$ starts in its invariant distribution, then a slight generalization of Theorem~\ref{general m} (see the Appendix for details) shows that the mean, $\E[u(x,t)]$, is the sum $v_{0}(x,t)+v_{1}(x,t)$, where
\begin{align}
&\partial_t\begin{pmatrix}v_0\\v_1\end{pmatrix}=D\Delta\begin{pmatrix}v_0\\v_1\end{pmatrix}+Q^{T}\begin{pmatrix}v_0\\v_1\end{pmatrix},\quad x\in(0,L/2)\cup(L/2,L)\text{, }t>0,\label{interface equation}\\
&v_0(0,t)=v_1(0,t)=0,
\quad
v_0(L,t)=\rho_{0}c,
\quad\text{and}\quad
v_1(L,t)=\rho_{1}c,\label{interface boundary}\\
&{v_0}_+={v_0}_-,
\quad
\partial_x{v_0}_+=\partial_x{v_0}_-,
\quad\text{and}\quad
\partial_x{v_1}_+=\partial_x{v_1}_-=0,\label{interface interface}
\end{align}
where $f_{\pm}:=\lim_{x\to L/2\pm}f(x)$. Starting with this BVP, it was proven that in a certain fast switching limit, the mean, $\E[u(x,t)]$, converges to the solution of the heat equation on $[0,L]$ with an interface jump condition at $x=L/2$ \cite{Lawley15robin}.}
\end{example}

\section{Particle perspective}\label{section particle}

\subsection{Hierarchy of joint exit statistics}

In this section, we study sets of particles that diffuse in a bounded domain. We suppose that some aspects of the environment randomly change according to a Markov jump process. This changing environment causes (a) the SDE governing the motion of each particle to change and (b) pieces of the boundary of the domain to switch between being absorbing or reflecting. Though the Brownian motions driving the diffusion of each particle are independent, the particle trajectories are correlated since the particles diffuse in the same randomly changing environment.

In Theorems~\ref{theorem survival}, \ref{theorem split}, and \ref{theorem mfpt}, we derive BVPs for various exit statistics of the particles. These BVPs are closely related to the BVPs given in Theorem~\ref{general m} for moments of solutions to randomly switching PDEs. The fact that \emph{some} connection exists between these two perspectives on diffusion in a random environment is not surprising. Indeed, connections between potential theory and Brownian motion have a long history \cite{doob54, kac51, kakutani44}. However, the correspondence elucidated here between moments of switching PDEs and joint exit statistics for multiple particles following a switching SDE was not expected. Further, this correspondence allows one to go back and forth between the two perspectives in order to exploit the advantages of each perspective \cite{lawley16pb5}.

First, we define the jump process controlling the switching environment. As in section~\ref{section partial} above, let $\{J(t)\}_{t\ge0}$ be a continuous time Markov jump process on a finite set $I$ with generator $Q$. Recall that the generator $Q$ is an $I\times I$ matrix with nonnegative off diagonal entries $q_{ij}\ge0$ giving the jump rate from state $i\in I$ to $j\in I$. The diagonal entries of $Q$ are chosen so that $Q$ has zero row sums and thus correspond to (minus) the total rate or leaving state $i\in I$.

Let $U\subset\R^{d}$ be a bounded open set with $C^{2}$ boundary. In order to describe the switching boundary, for each $i\in I$ let the disjoint sets $\Gamma_{i}^{\text{abs}}$ and $\Gamma_{i}^{\text{ref}}$ partition the boundary
\begin{align*}
\Gamma_{i}^{\text{abs}}\cup\Gamma_{i}^{\text{ref}}=\partial U.
\end{align*}
If $J(t)=i$, then $\Gamma_{i}^{\text{abs}}$ is absorbing and $\Gamma_{i}^{\text{ref}}$ is reflecting. In contrast to section~\ref{section partial} above where we considered Dirichlet, Neumann, and Robin boundary conditions for a PDE, here we consider only absorbing and reflecting boundary conditions for an SDE. The reason for this disparity is that while absorbing and reflecting conditions are SDE analogues of Dirichlet and Neumann conditions, the SDE analog of a Robin condition is significantly more complicated (see \cite{ito12} section 2.3).

Suppose there are $M$ particles that begin at positions $\x_{1},\dots,\x_{M}\in\bar{U}$. For each $m\in\{1,\dots,M\}$, let $\X_{m}(t)$ denote the position of the $m$-th particle at time $t\ge0$ and let $\tau_{m}$ be the first time that the $m$-th particle hits an absorbing piece of the boundary. That is, $\tau_{m}$ is the stopping time\footnote{A stopping time is a random variable whose value is interpreted as the time when a given stochastic process is terminated according to some rule that depends on current and past states. A classical example of a stopping time is a first passage time.}
\begin{align*}
\tau_{m}:=\inf\big\{t\ge0:\X_{m}(t)\in\Gamma_{J(t)}^{\text{abs}}\big\},
\end{align*}
which we assume to be finite almost surely. For $t>\tau_{m}$, we set $\X_{m}(t)=\X_{m}(\tau_{m})$ and say that the particle has exited the domain. For $t\le\tau_{m}$, we assume that $\{\X_{m}(t)\}_{t\ge0}$ is the unique solution to
\begin{align}\label{SDE3}
d\X_{m}(t)&=\bbb_{J(t)}(\X_{m}(t))\,dt+\sigma_{J(t)}(\X_{m}(t))\,d\W_{m}(t)+\n(\X_{m}(t))\,dK_{m}(t),
\end{align}
with $\X_{m}(0)=\x_{m}$, where $\{\bbb_{i}\}_{i\in I}$ and $\{\sigma_{i}\}_{i\in I}$ are given Lipschitz functions
\begin{align*}
\bbb_{i}:\bar{U}\mapsto\R^{d}
\quad\text{and}\quad
\sigma_{i}:\bar{U}\mapsto\R^{d\times d},
\end{align*}
$\W_{m}(t)$ is an $\R^{d}$-valued standard Brownian motion, $\n:\partial U\mapsto\R^{d}$ is the inner normal field, and $K_{m}(t)$ is the \emph{local time} of $\X_{m}(t)$ in $\partial U$. The local time is the time that $\X_{m}(t)$ spends on $\partial U$. Precisely, $K_{m}(t)$ is non-decreasing and increases only when $\X_{m}(t)$ is in $\partial U$ and $K_{m}(0)=0$. The significance of the local time term in (\ref{SDE3}) is that it forces $\X_{m}$ to reflect from $\partial U$ in the normal direction and thus ensures that $\X_{m}(t)\in\bar{U}$ for all $t\ge0$.

We assume that $\{\W_{1}(t)\}_{t\ge0}$, \dots, $\{\W_{M}(t)\}_{t\ge0}$, and $\{J(t)\}_{t\ge0}$ are independent. Though these driving Brownian motions, $\{\W_{1}(t)\}_{t\ge0},\dots,\{\W_{M}(t)\}_{t\ge0}$, are independent, the trajectories, $\{\X_{1}(t)\}_{t\ge0},\dots,\{\X_{M}(t)\}_{t\ge0}$, are nonetheless correlated since they all experience the same changing environment (that is, the same $J(t)$).

We note that for each $m\in\{1,\dots,M\}$, the pair $(\X_{m}(t),J(t))$ is a strong Markov process\footnote{A stochastic process is a Markov process if the conditional probability distribution of future states of the process (conditioned on both past and present states) depends only upon the present state, not on the sequence of states that preceded it. A strong Markov process is similar to a Markov process, except that the  ``present'' is defined in terms of a stopping time.} and is commonly known as a hybrid switching diffusion (see the books \cite{Yin10} and \cite{mao06} for more information about hybrid switching diffusions).

The following theorem gives the survival probability of at least one of the $M$ particles.

\begin{theorem}[Survival probability]\label{theorem survival}
Let $M$ be a positive integer. For each $\M\in\{1,\dots,M\}$, assume that the functions $\{p^{\M}_{i}(\x_{1},\dots,\x_{\M},t)\}_{i\in I}$
\begin{align*}
p^{\M}_{i}:\bar{U}^{\M}\times[0,\infty)\to[0,1]
\end{align*}
are continuously differentiable in $t$ and twice continuously differentiable in their other argument and satisfy the PDE
\begin{align}\label{430p}
\begin{split}
\partial_{t}p_{i}^{\M}&=\mathbb{L}^{\M}p_{i}^{\M}\quad\text{for }(\x_{1},\dots,\x_{\M},t)\in U^{\M}\times(0,\infty),
\end{split}
\end{align}
where $\mathbb{L}^{\M}$ is the operator
\begin{align}\label{bb}
\mathbb{L}^{\M}p_{i}^{\M}=\sum_{m=1}^{\M} \L^{m}_{i}p_{i}^{\M}+\sum_{j\in I} q_{ij}p^{\M}_j,
\end{align}
where $q_{ij}$ is the $(i,j)$-th entry of the generator $Q$ of $J(t)$, and $\L^{m}_{i}$ is the following differential operator acting on the $m$-th spatial variable $\x_{m}$
\begin{align*}
\L^{m}_{i}p_{i}^{\M} & = \sum_{j=1}^{d}(\bbb_{i})_{j}(\x_{m})\partial_{(\x_{m})_{j}}p_{i}^{\M}+\frac12\sum_{l,j=1}^{d}(\sigma_{i}\sigma_{i}^{T})_{l,j}(\x_{m})\partial_{(\x_{m})_{l},(\x_{m})_{j}}p_{i}^{\M}.
\end{align*}

Assume that the following boundary conditions are satisfied for each $\M\in\{1,\dots,M\}$
\begin{alignat}{2}
&\partial_{\n_{m}}p_{i}^{\M}(\x_{1},\dots,\x_{\M},t)=0,\quad &&\x_{m}\in\Gamma_{i}^{\text{ref}},\label{431p}\\
&p_{i}^{\M}(\x_{1},\dots,\x_{\M},t)=p_{i}^{\M-1}(\x_{1},\dots,\x_{m-1},\x_{m+1},\dots,\x_{\M},t),\quad &&\x_{m}\in\Gamma_{i}^{\text{abs}},\label{431p2}
\end{alignat}
where $\partial_{\n_{m}}$ denotes the normal derivative with respect to the $m$-th spatial variable $\x_{m}$ and $p_{i}^{0}\equiv0$ for each $i$.

Assume that the following initial conditions are satisfied for each $\M\in\{1,\dots,M\}$
\begin{align}\label{p initial}
\begin{split}
p_{i}^{\M}(\x_{1},\dots,\x_{\M},0)=1\quad\text{for }(\x_{1},\dots,\x_{\M})\in U^{\M}.
\end{split}
\end{align}

Then $p_{j}^{M}(\x_{1},\dots,\x_{M},t)$ gives the probability that at least one of the $M$ particles is still in the interior of the domain at time $t$ given that they start at positions $\x_{1},\dots,\x_{M}$ and $J(0)=j$. That is, $p^{M}_j(\x_{1},\dots,\x_{M},t)$ is equal to
\begin{align*}
\P\big(\cup_{m=1}^{M}\{\X_{m}(t)\in U\}\,\big|\, \X_{1}(0)=\x_{1},\dots,\X_{M}(0)=\x_{M},\,J(0)=j\big).
\end{align*}
\end{theorem}

\begin{remark}
{\rm In matrix notation, the PDE in (\ref{430p}) is
\begin{align*}
\partial_{t}{\bf p}=(L+Q){\bf p},
\end{align*}
where ${\bf p}$ is the vector with $i$-th component $p_{i}^{\M}$, $L$ is the diagonal matrix with $i$-th diagonal entry $\sum_{m=1}^{\M}\L_{i}^{m}$ and $Q$ is the generator of $J(t)$. Similar statements hold for the PDEs in Theorems~\ref{theorem split} and \ref{theorem mfpt} below.}
\end{remark}

\begin{remark}
{\rm If we make (\ref{431p2}) an absorbing condition, then the solution $p^{M}_{j}$ gives the probability that \emph{all} $M$ particles are still in the interior of the domain at time $t$.}
\end{remark}

\begin{remark}
{\rm If instead of a switching environment, we impose that each particle switches independently, then calculating joint statistics requires only solving a single BVP on $U$, instead of the hierarchy of $M$ BVPs on $U$, $U^{2}$,\dots, $U^{M}$ given in the theorem above.}
\end{remark}

\begin{proof}
We prove the theorem by induction on the number of particles, so we first consider the base case of $M=1$. Let $S>0$, define $q(\x,t,j):=p_{j}^{1}(\x,S-t)$, and let $\E_{0}$ denote expectation conditioned on $\X_{1}(0)=\x_{1}$ and $J(0)=j$. The generalized Ito formula\footnote{Ito's formula is a fundamental result in stochastic analysis and is the stochastic counterpart to the chain rule \cite{oksendal03}. Here, we use the generalized Ito formula which applies to SDEs with random switching. For for information, see Lemma 3 on p.\ 104 of \cite{skorokhod89} or Lemma 1.9 on p.\ 49 of \cite{mao06}.} gives
\begin{align}\label{ito p}
\begin{split}
&\E_{0}\Big[q(\X_{1}(S\w\tau_{1}),S\w\tau_{1},J(S\w\tau_{1}))\Big] - q(\x_{1},0,j)\\
& = \E_{0}\Big[\int_0^{S\w\tau_{1}} (\partial_{t}+\mathbb{L}^{1})q(\X_{1}(s),s,J(s))\,ds \Big]
-\E_{0}\Big[\int_0^{S\w\tau_{1}} \partial_{\n} q(\X_{1}(s),s,J(s))\,dK_{1}(s) \Big].
\end{split}
\end{align}
The PDE in (\ref{430p}), the no flux boundary conditions in (\ref{431p}), and the definitions of $\tau_{1}$ and $q$ ensure that the righthand side of (\ref{ito p}) is zero. Hence,
\begin{align}\label{q equation}
q(\x_{1},0,j)=\E_{0}\big[q(\X_{1}(S\w\tau_{1}),S\w\tau_{1},J(S\w\tau_{1}))\big].
\end{align}
Recalling the definition of $q$, equation~(\ref{q equation}) becomes
\begin{align}\label{72}
p_{j}^{1}(\x_{1},S)
&=\E_{0}\big[p_{J(S)}^{1}(\X_{1}(S),0)1_{S<\tau_{1}}\big]
+\E_{0}\big[p_{J(\tau_{1})}^{1}(\X_{1}(\tau_{1}),S-\tau_{1})1_{\tau_{1}\le S}\big],
\end{align}
where $1_{A}$ denotes the indicator function on an event $A$. By definition of $\tau_{1}$, we have that $\X_{1}(\tau_{1})\in\partial U$ and thus by the boundary condition in (\ref{431p2}) we have that the second term in the righthand side of (\ref{72}) is zero. Therefore, by the initial condition in (\ref{p initial}) and the definition of $\tau_{1}$, we have that (\ref{72}) becomes
\begin{align*}
p_{j}^{1}(\x_{1},S)=\P(\X_{1}(S)\in U\,|\,\X_{1}(0)=\x_{1},J(0)=j),
\end{align*}
which completes the proof for $M=1$.

Now suppose $M\ge2$ and let $\tau$ be the time that the first particle exits the domain
\begin{align*}
\tau:=\inf_{1\le m\le M}\tau_{m}.
\end{align*}
Let $S>0$, define $q(\y_{1},\dots,\y_{M},t,j):=p_{j}^{M}(\y_{1},\dots,\y_{M},S-t)$, and let $\E_{0}$ denote expectation conditioned on $\X_{1}(0)=\x_{1}$, \dots, $\X_{M}(0)=\x_{M}$, and $J(0)=j$. Again by the generalized Ito formula we have that
\begin{align}\label{ito2p}
\begin{split}
&\E_{0}\Big[q(\X_{1}(S\w\tau),\dots,\X_{M}(S\w\tau),S\w\tau,J(S\w\tau))\Big] - q(\x_{1},\dots,\x_{M},0,j) \\
& \quad= \E_{0}\Big[\int_0^{S\w\tau} (\partial_{t}+\mathbb{L}^{M})q(\X_{1}(s),\dots,\X_{M}(s),s,J(s))\,ds \Big] 
\\
&\quad\quad- \sum_{m=1}^{M}\E_{0}\Big[\int_0^{S\w\tau} \partial_{\n_{m}} q(\X_{1}(s),\dots,\X_{M}(s),s,J(s))\,dK_{m}(s) \Big].
\end{split}
\end{align}
As before, by (\ref{430p}), (\ref{431p}), and the definitions of $\tau$ and $q$, we have that the righthand side of (\ref{ito2p}) is zero and thus
\begin{align}\label{194}
q(\x_{1},\dots,\x_{M},0,j)=\E_{0}\Big[q(\X_{1}(S\w\tau),\dots,\X_{M}(S\w\tau),S\w\tau,J(S\w\tau))\Big].
\end{align}
Recalling the definition of $q$, equation~(\ref{194}) becomes
\begin{align}
\begin{split}\label{s}
p_{j}^{M}(\x_{1},\dots,\x_{M},S)
&=\E_{0}\Big[p^{M}_{J(S)}(\X_{1}(S),\dots,\X_{M}(S),0)1_{S<\tau}\Big]\\
&\qquad+\E_{0}\Big[p^{M}_{J(\tau)}(\X_{1}(\tau),\dots,\X_{M}(\tau),S-\tau)1_{\tau\le S}\Big].
\end{split}
\end{align}
By the initial condition in (\ref{p initial}) and the definition of $\tau$, the first term on the righthand side of (\ref{s}) is the probability that none of the $M$ particles exit before time $S$
\begin{align*}
\P\big(\cap_{m=1}^{M}\{\X_{m}(t)\in U\}\,\big|\, \X_{1}(0)=\x_{1},\dots,\X_{M}(0)=\x_{M},\,J(0)=j\big).
\end{align*}
By (\ref{431p2}), the inductive hypothesis, and the strong Markov property, the second term is the probability that the number of particles that exit before time $S$ is between 1 and $M-1$. Summing these two terms completes the proof.
\end{proof}

The following theorem gives the probability that all $M$ particles exit the domain through the same piece of the boundary.

\begin{theorem}[Hitting probability]\label{theorem split}
Let $M$ be a positive integer. For each $\M\in\{1,\dots,M\}$, assume that the functions $\{\pi^{\M}_{i}(\x_{1},\dots,\x_{\M})\}_{i\in I}$
\begin{align*}
\pi^{\M}_{i}:\bar{U}^{\M}\to[0,1]
\end{align*}
are twice continuously differentiable solutions to the PDE
\begin{align}\label{4302}
\begin{split}
0&=\mathbb{L}^{\M}\pi_{i}^{\M}\quad\text{for }(\x_{1},\dots,\x_{\M})\in U^{\M},
\end{split}
\end{align}
where $\mathbb{L}^{\M}$ is the operator defined in (\ref{bb}).

For some given $\Gamma_{0}\subset\partial U$, assume that the following boundary conditions are satisfied for each $\M\in\{1,\dots,M\}$
\begin{alignat}{2}
&\partial_{\n_{m}}\pi_{i}^{\M}(\x_{1},\dots,\x_{\M})=0,\quad &&\x_{m}\in\Gamma^{\text{ref}}_{i},\label{4312}\\
&\pi_{i}^{\M}(\x_{1},\dots,\x_{\M})=0,\quad &&\x_{m}\in\Gamma^{\text{abs}}_{i}\slash\Gamma_{0},\label{43122}\\
\pi_{i}^{\M}&(\x_{1},\dots,\x_{\M})=\pi_{i}^{\M-1}(\x_{1},\dots,\x_{m-1},\x_{m+1},\dots,\x_{\M}),\quad &&\x_{m}\in\Gamma^{\text{abs}}_{i}\cap\Gamma_{0},\label{43123}
\end{alignat}
where $\partial_{\n_{m}}$ denotes the normal derivative with respect to the $m$-th spatial variable $\x_{m}$, and $\pi_{i}^{0}\equiv1$ for each $i$.

Then $\pi_{j}^{M}(\x_{1},\dots,\x_{M})$ gives the probability that all $M$ particles exit through $\Gamma_{0}\subset\partial U$ given that they start at positions $\x_{1},\dots,\x_{M}$ and $J(0)=j$. That is, $\pi^{M}_j(\x_{1},\dots,\x_{M})$ is equal to
\begin{align*}
\P\big(\cap_{m=1}^{M}\big\{\lim_{t\to\infty}\X_{m}(t)\in\Gamma_{0}\big\}\big|\, \X_{1}(0)=\x_{1},\dots,\X_{M}(0)=\x_{M},\,J(0)=j\big).
\end{align*}
\end{theorem}

\begin{proof}
We prove the theorem by induction on the number of particles, so we first consider the base case of $M=1$. Denote $\pi^{1}_{j}(\x)$ by $\pi(\x,j)$ and let $\E_{0}$ denote expectation conditioned on $\X_{1}(0)=\x_{1}$ and $J(0)=j$. By the generalized Ito formula, we have that
\begin{align}\label{ito}
\begin{split}
&\E_{0}\big[\pi(\X_{1}(t\w\tau_{1}),J(t\w\tau_{1}))\big] - \pi(\x_{1},j)\\
& = \E_{0}\Big[\int_0^{t\w\tau_{1}} \mathbb{L}^{1}\pi(\X_{1}(s),J(s))\,ds \Big]
 - \E_{0}\Big[\int_0^{t\w\tau_{1}} \partial_{\n} \pi(\X_{1}(s),J(s))\,dK_{1}(s) \Big].
\end{split}
\end{align}
The PDE in (\ref{4302}), the no flux boundary conditions in (\ref{4312}), and the definition of $\tau_{1}$ ensure that the righthand side of (\ref{ito}) is zero. Hence,
\begin{align*}
\pi(\x_{1},j)=\E_{0}\big[\pi(\X_{1}(t\w\tau_{1}),J(t\w\tau_{1}))\big].
\end{align*}
Taking $t\to\infty$ and consulting (\ref{43122}) and (\ref{43123}) completes the proof for $M=1$.

Now suppose $M\ge2$ and let $\tau$ be the time that the first particle exits the domain
\begin{align*}
\tau:=\inf_{1\le m\le M}\tau_{m}.
\end{align*}
Denote $\pi^{M}_{j}(\y_{1},\dots,\y_{M})$ by $\pi(\y_{1},\dots,\y_{M},j)$ and let $\E_{0}$ denote expectation conditioned on $\X_{1}(0)=\x_{1}$, \dots, $\X_{M}(0)=\x_{M}$, and $J(0)=j$. Again by the generalized Ito formula we have that
\begin{align}\label{ito2}
\begin{split}
&\E_{0}\big[\pi(\X_{1}(t\w\tau),\dots,\X_{M}(t\w\tau),J(t\w\tau))\big] - \pi(\x_{1},\dots,\x_{M},j) \\
& \quad= \E_{0}\Big[\int_0^{t\w\tau} \mathbb{L}^{M}\pi(\X_{1}(s),\dots,\X_{M}(s),J(s))\,ds \Big] 
\\
&\quad\quad- \sum_{m=1}^{M}\E_{0}\Big[\int_0^{t\w\tau} \partial_{\n_{m}} \pi(\X_{1}(s),\dots,\X_{M}(s),J(s))\,dK_{m}(s) \Big].
\end{split}
\end{align}
As before, by (\ref{4302}), (\ref{4312}), and the definition of $\tau$, we have that (\ref{ito2}) becomes
\begin{align*}
\pi(\x_{1},\dots,\x_{M},j)=\E_{0}\big[\pi(\X_{1}(t\w\tau),\dots,\X_{M}(t\w\tau),J(t\w\tau))\big].
\end{align*}
Taking $t\to\infty$, consulting (\ref{43122}) and (\ref{43123}), and using the strong Markov property and the inductive hypothesis completes the proof.
\end{proof}

A similar argument gives the mean first passage time of the last particle to exit. 
\begin{theorem}[MFPT of last particle]\label{theorem mfpt}
Let $M$ be a positive integer. For each $\M\in\{1,\dots,M\}$, assume that the functions $\{w^{\M}_{i}(\x_{1},\dots,\x_{\M})\}_{i\in I}$
\begin{align*}
w^{\M}_{i}:\bar{U}^{\M}\to[0,\infty)
\end{align*}
are twice continuously differentiable solutions to the PDE\begin{align*}
\begin{split}
-1&=\mathbb{L}^{\M}w_{i}^{\M}\quad\text{for }(\x_{1},\dots,\x_{\M})\in U^{\M},
\end{split}
\end{align*}
where $\mathbb{L}^{\M}$ is the operator defined in (\ref{bb}).

Assume that the following boundary conditions are satisfied for each $\M\in\{1,\dots,M\}$
\begin{alignat*}{2}
\partial_{\n_{m}}w_{i}^{\M}(\x_{1},\dots,\x_{\M})&=0,\quad &&\x_{m}\in\Gamma_{i}^{\text{ref}},\\
w_{i}^{\M}(\x_{1},\dots,\x_{\M})&=w_{i}^{\M-1}(\x_{1},\dots,\x_{m-1},\x_{m+1},\dots,\x_{\M}),\quad &&\x_{m}\in\Gamma_{i}^{\text{abs}},
\end{alignat*}
where $\partial_{\n_{m}}$ denotes the normal derivative with respect to the $m$-th spatial variable $\x_{m}$ and $w_{i}^{0}\equiv0$ for each $i$.

Then $w_{j}^{M}(\x_{1},\dots,\x_{M})$ gives the mean first passage time of the last of $M$ particles to exit the domain given that they start at positions $\x_{1},\dots,\x_{M}$ and $J(0)=j$. That is, if $\S:=\sup_{1\le m\le M}\tau_{m}$, then
\begin{align*}
w^{M}_j(\x_{1},\dots,\x_{M})
&\;=\E[\S|\, \X_{1}(0)=\x_{1},\dots,\X_{M}(0)=\x_{M},\,J(0)=j].
\end{align*}
\end{theorem}

\begin{proof}
The proof is analogous to the proof of Theorem~\ref{theorem split}.
\end{proof}

\subsection{SDE example}\label{SDE examples}

\begin{example}[Gated target versus gated ligands]\label{particle example} {\rm As some cellular reactions depend on the arrival of diffusing ligands to small targets, many works seek to calculate the mean first passage time of a diffusing particle to a small target (the so-called ``narrow escape problem'' \cite{cheviakov10, Holcman14, lindsay15, pillay10, schuss07}). If, however, the diffusing ligands or the target change conformational state and reaction is only possible in some states, then the theory must be adjusted \cite{Bressloff15b, Bressloff15c, Bressloff15d}. Indeed, such reactions (known as ``gated'' reactions) occur in a number of biological and biochemical contexts, including medical therapies that block chemical reactions \cite{spouge96}, diffusing enzymes that switch between an active and an inactive state, the binding of a transcription factor to a DNA promoter \cite{barrandon08}, 
and the diffusion of ions through stochastically gated channels \cite{Newby13}. Intermittent search processes can also fit into this framework \cite{Benichou11}. 

The situation gets more interesting if there are multiple diffusing ligands. If there is only one diffusing ligand, then it does not matter if it is the state of the ligand or of the target that determines the possibility of reaction. However, if there is more than one ligand, then these two cases become significantly different. This difference was first pointed out in \cite{zhou96} and further investigated in \cite{Berezhkovskii96, Makhnovskii98}. The key difference is that if the target changes state, then all the ligands become correlated even though they move independently. Our theorems in section~\ref{section particle} for multiple diffusing particles apply to this more delicate case.

As a prototype model, consider $M$ non-interacting ligands that move by pure diffusion in the interval $[0,L]$ with an absorbing boundary condition at $x=0$. Suppose that each ligand can bind to a stationary protein at $x=L$ if the protein is in the proper conformational state. Suppose that the state of the protein is determined by a continuous-time Markov jump process $J(t)\in\{0,1\}$. State 0 is the binding state, and thus all $M$ ligands have an absorbing condition at $x=L$ if $J(t)=0$. If $J(t)=1$, then all $M$ ligands reflect at $x=L$. 

Though the ligands are non-interacting, they are nonetheless correlated because they all diffuse in the presence of the same switching protein. Calculating joint statistics for the $M$ ligands requires solving a hierarchy of $M$ BVPs on the hypercubes $[0,L],[0,L]^{2},\dots,$ and $[0,L]^{M}$, where the BVPs couple to each other at the boundaries.  To illustrate, suppose the ligands begin at positions $x_{1},\dots,x_{M}$ and $J(0)=0$. By Theorem~\ref{theorem survival}, if $Q$ is the generator of $J(t)$, then the probability that the $m$-th ligand has not been absorbed by time $t$ is given by $p_{0}^{(1)}(x_{m},t)$, where
\begin{align*}
p^{(1)}_{0}:[0,L]\times[0,\infty)\to[0,1]
\quad\text{and}\quad
p^{(1)}_{1}:[0,L]\times[0,\infty)\to[0,1]
\end{align*}
satisfy
\begin{align}
&\partial_t\begin{pmatrix}p^{(1)}_{0}\\p^{(1)}_{1}\end{pmatrix}=D\Delta\begin{pmatrix}p^{(1)}_{0}\\p^{(1)}_{1}\end{pmatrix}+Q\begin{pmatrix}p^{(1)}_{0}\\p^{(1)}_{1}\end{pmatrix},\quad y\in(0,L)\text{, }t>0,\label{one equation}\\
&p^{(1)}_{0}(0,t)=p^{(1)}_{1}(0,t)=p^{(1)}_{0}(L,t)=\partial_{x}p^{(1)}_{1}(L,t)=0,\quad t>0,\label{one boundary}\\
&p^{(1)}_{0}=p^{(1)}_{1}=1,\quad y\in(0,1)\text{, }t=0.\label{one initial}
\end{align}

Using Theorem~\ref{theorem survival} again, the probability that the either the $m$-th or the $n$-th ligand (or both) has not been absorbed by time $t$ is given by $p_{0}^{(2)}(x_{m},x_{n},t)$, where
\begin{align*}
p^{(2)}_{0}:[0,L]^{2}\times[0,\infty)\to[0,1]
\quad\text{and}\quad
p^{(2)}_{1}:[0,L]^{2}\times[0,\infty)\to[0,1]
\end{align*}
satisfy
\begin{align*}
&\partial_t\begin{pmatrix}p^{(2)}_{0}\\p^{(2)}_{1}\end{pmatrix}=D\Delta\begin{pmatrix}p^{(2)}_{0}\\p^{(2)}_{1}\end{pmatrix}+Q\begin{pmatrix}p^{(2)}_{0}\\p^{(2)}_{1}\end{pmatrix},\quad (y_{1},y_{2})\in(0,L)\times(0,L)\text{, }t>0,
\end{align*}
with some absorbing and reflecting boundary conditions
\begin{align*}
p^{(2)}_{0}(0,y_{2},t)=p^{(2)}_{1}(0,y_{2},t)=p^{(2)}_{1}(y_{1},0,t)=p^{(2)}_{0}(y_{1},0,t)=0,\\
\partial_{y_{1}}p^{(2)}_{1}(L,y_{2},t)=\partial_{y_{2}}p^{(2)}_{1}(y_{1},L,t)=0,
\end{align*}
and boundary conditions that couple to $p_{0}^{(1)}$
\begin{align*}
p^{(2)}_{0}(L,y_{2},t)=p^{(1)}_{0}(y_{2},t),
\quad\text{and}\quad
p^{(2)}_{0}(y_{1},L,t)=p^{(1)}_{0}(y_{1},t),
\end{align*}
and initial conditions given in (\ref{one initial}).

Continuing in this manner, Theorem~\ref{theorem survival} gives that the probability that at least one of the $M$ ligands has not been absorbed by time $t$ is $p^{(M)}_{0}(x_{1},\dots,x_{M},t)$, where
\begin{align*}
p^{(M)}_{0}:[0,L]^{M}\times[0,\infty)\to[0,1]
\quad\text{and}\quad
p^{(M)}_{1}:[0,L]^{M}\times[0,\infty)\to[0,1]
\end{align*}
satisfy
\begin{align*}
&\partial_t\begin{pmatrix}p^{(M)}_{0}\\p^{(M)}_{1}\end{pmatrix}=D\Delta\begin{pmatrix}p^{(M)}_{0}\\p^{(M)}_{1}\end{pmatrix}+Q\begin{pmatrix}p^{(M)}_{0}\\p^{(M)}_{1}\end{pmatrix},\quad (y_{1},\dots,y_{M})\in(0,L)^{M}\text{, }t>0,
\end{align*}
with some boundary conditions that couple to $p^{(M-1)}_{0}$, which solves a similar BVP on $[0,L]^{M-1}$ (see Theorem~\ref{theorem survival} for a precise statement). Theorems~\ref{theorem split} and \ref{theorem mfpt} give hierarchies of BVPs for other joint statistics, and Theorem~\ref{general m} gives similar hierarchies of BVPs for moments of the switching PDEs in Examples~\ref{neural example 1} and \ref{insect example}-\ref{robin example}.}

\end{example}

\section{Discussion}
When studying diffusion in a randomly switching environment, considering a density of diffusing particles leads to a switching PDE. In contrast, considering only finitely many diffusing particles leads to a switching SDE. In this paper we developed tools to calculate statistics for both of these types of processes and have shown how these tools reveal the dynamics of several biological examples. We have also established a connection between these two perspectives on diffusion in a random environment. In particular, moments of switching PDEs correspond to exit statistics of multiple diffusing particles. A number of forthcoming papers on diverse subjects depend on the tools developed in this paper, and we further anticipate that more models involving switching PDEs and SDEs will arise and make use of our results.

\setcounter{equation}{0}
\renewcommand{\theequation}{A.\arabic{equation}}
\section*{Appendix}
In this appendix, we discuss verifying that the examples in section~\ref{PDE examples} satisfy the necessary hypotheses. First, we consider Examples~\ref{neural example 1} and \ref{insect example} together. For both of these examples, the existence of such a process satisfying Assumptions~\ref{continuous space}-\ref{boundary assumption} follows immediately from the regularity of solutions to the one-dimensional heat equation on finite intervals (for a detailed construction of such a process, see \cite{Lawley15sima}). The existence of the bound in Assumption~\ref{C bound} follows from standard estimates on solutions to the heat equation (for example, Theorem 9 on page 61 of \cite{Evans98} combined with the maximum principle gives such a bound).

Assumption~\ref{neumann bound} is verified by analyzing the spectral decompositions of the associated solution operators. For concreteness, consider Example~\ref{neural example 1}. Fix a time $t>0$ and let $\sigma\ge0$ denote the amount of time since the last switch (known in renewal theory as the age). Then, $u(x,t)$ can be written as
\begin{align}\label{any time}
\begin{split}
u(x,t)&=(1-J(t))e^{A_{q}\sigma}u(x,t-\sigma)\\
&\quad+J(t)(e^{A_{f}\sigma}(u(x,t-\sigma)-h) + \int_{0}^{\sigma}e^{A_{f}(\sigma-s)}Dh_{xx}\,ds+h),
\end{split}
\end{align}
where $e^{A_{q}t}$ and $e^{A_{f}t}$ are the $C_{0}$-semigroups  generated by the self-adjoint operators
\begin{align*}
A_{q}u & := \Delta u\quad\text{if }u\in D(A_{q})  := \Big\{\phi\in H^{2}(0,L):\frac{\partial\phi}{\partial x}(0)=0=\phi(L)\Big\},\\
A_{f}u & := \Delta u\quad\text{if }u\in D(A_{f})  := \Big\{\phi\in H^{2}(0,L):\frac{\partial\phi}{\partial x}(0)=0=\frac{\partial\phi}{\partial x}(L)\Big\},
\end{align*}
and $h(x):=\frac{2L}{\pi}\delta[1-\cos(\frac{\pi}{2L}x)]$.

If we let $\{-\alpha_{k}\}_{k\ge1}$ and $\{a_{k}\}_{k\ge1}$ denote the eigenvalues and eigenvectors of $A_{q}$, then we have the almost sure bound
\begin{align*}
\big\|\frac{d}{dx}e^{A_{q}\sigma}u(x,t-\sigma)\big\|_{\infty}
\le\big\|a_{1}\big\|_{\infty}\sum_{k=1}^{\infty}e^{-\alpha_{k}\sigma}|\<a_{k},u(x,t-\sigma)\>|\sqrt{\alpha_{k}/D},
\end{align*}where $\<\cdot,\cdot\>$ denotes the $L^{2}[0,L]$ inner product and $\|\cdot\|_{\infty}$ denotes the $L^{\infty}[0,L]$ norm. Using the eigendecompositions of $A_{q}$ and $A_{f}$, it is straightforward to show that
\begin{align*}
\E e^{-\alpha_{k}\sigma}|\<a_{k},u(x,t-\sigma)\>|\sim1/k^{3}\quad\text{as }k\to\infty.
\end{align*}
A similar argument shows that there exists a random variable with finite expectation that almost surely bounds the second term in (\ref{any time}). Verifying Assumption~\ref{neumann bound} for Example~\ref{insect example} is similar.

Moving to Example~\ref{thermostat}, checking Assumptions~\ref{continuous space}-\ref{boundary assumption} is the same as the examples above. Trying to verify Assumptions~\ref{C bound} and \ref{neumann bound} is more difficult since the boundary conditions are non-local and the operators involved are not self-adjoint. We currently do not know how to verify these assumptions, but we remark that the conclusions of Theorem~\ref{general m} are in complete agreement with Monte Carlo simulations (see Figure~\ref{figure thermo}).

Example~\ref{robin example} requires only a slight generalization of Theorem~\ref{general m}. The PDE in~(\ref{interface equation}) follows from exactly the same argument used in Theorem~\ref{general m} to exchange differentiation with expectation (the necessary bound in Assumption~\ref{C bound} is obtained again by Theorem 9 on page 61 of \cite{Evans98} combined with the maximum principle). The boundary conditions in~(\ref{interface boundary}) are immediate. The interface conditions in (\ref{interface interface}) follow from exchanging limits with expectation, the necessary bounds coming from an argument similar to the one used above to verify Assumption~\ref{neumann bound} for Example~\ref{neural example 1}.

\bibliography{ckbiblio}
\bibliographystyle{siam}
\end{document}